\documentclass[11pt]{amsart}
\usepackage{stmaryrd}
\usepackage{xcolor}
\usepackage{epsfig}
\usepackage{enumerate}
\usepackage{comment}
\usepackage{amsmath}
\usepackage{amssymb}
\usepackage{amsthm}
\usepackage{amscd}
\usepackage{graphicx}
\usepackage{lineno}
\usepackage{pstricks}
\usepackage{tocvsec2}
\usepackage{mathtools}
\usepackage{todonotes}
\usepackage[basic]{mnotes}
\allowdisplaybreaks[4]

\usepackage{rotating}

\usepackage[colorlinks, citecolor=red, anchorcolor=black, linkcolor=blue]{hyperref}

\usepackage{amsfonts}
\usepackage[top=35mm, bottom=35mm, left=30mm, right=30mm]{geometry}
\usepackage{verbatim}
\usepackage{graphicx, amssymb, amsmath,amsthm}
\allowdisplaybreaks[4]
\usepackage{appendix}
\usepackage{mathrsfs}
\usepackage{cite}
\usepackage{titlesec}
\usepackage{titletoc}
\titleformat{\part}[block]{\huge\bfseries\centering}{Part \arabic{part}.}{0.5em}{}[]
\titleformat{\section}[block]{\Large\bfseries\centering}{ \arabic{section}.}{1em}{}[]
\titleformat{\subsection}[block]{\mdseries\bfseries}{\arabic{section}.\arabic{subsection}.}{1em}{}[]
\titleformat{\subsubsection}[block]{\normalsize\bfseries}{\arabic{section}.\arabic{subsection}.\arabic{subsubsection}.}{1em}{}[]

\makeatletter
\@namedef{subjclassname@2020}{%
	\textup{2020} Mathematics Subject Classification}
\makeatother

\makeatletter
\newcommand\@notni[2]{\mathrel{\rotatebox[y=#1]{180}{$#2\notin$}}}
\newcommand\notni{
\mathchoice
  {\@notni{0.57ex}\displaystyle}
  {\@notni{0.57ex}\textstyle}
  {\@notni{0.39ex}\scriptstyle}
  {\@notni{0.26ex}\scriptscriptstyle}
}
\makeatother

\def \a{\alpha}   
   
 \def \l{\lambda}  
 \def \D{\Delta}  
  \def \L{\Lambda}

\def \L{\Lambda}
\def \N{\mathbb N}
\def \Z{\mathbb Z}

\def\odo{ \Sigma_{(p_n)}}
\usepackage{xcolor}

\usepackage{tikz}
\usetikzlibrary{decorations.pathreplacing,calc}

\newcommand{\tikzmark}[1]{\tikz[remember picture,baseline=(#1.base)] \node[inner sep=0pt, outer sep=0pt] (#1) {};}

\usepackage{xcolor}
\newtheorem{thm}{Theorem}[section]
\newtheorem{lem}[thm]{Lemma}
\newtheorem{cor}[thm]{Corollary}

\newtheorem{prop}[thm]{Proposition}

\theoremstyle{definition}
\newtheorem{defn}[thm]{Definition}
\newtheorem{rem}[thm]{Remark}
\newtheorem{fact}{Fact}%[section]

\numberwithin{equation}{section}
\usepackage{epsfig}

\newcommand{\good}{convenient}

\newtheorem{introthm}{Theorem}

\dottedcontents{part}[0.5em]{\normalsize}{1em}{4pt}
\dottedcontents{section}[1.5em]{\normalsize}{1.5em}{4pt}
\dottedcontents{subsection}[3.5em]{\small}{2.0em}{4pt}
\dottedcontents{subsubsection}[4em]{\normalsize}{3.0em}{4pt}
\numberwithin{equation}{section}

\renewcommand{\paragraph}[1]{\medskip\noindent\textbf{#1.}\quad}

\makeatother


	\author{Gabriel Fuhrmann}
	\address[Gabriel Fuhrmann]
	{Department of Mathematical Sciences, Durham University, DH1 3LE, United Kingdom.}
	\email{gabriel.fuhrmann@durham.ac.uk}

\author[Chunlin Liu]{Chunlin Liu\textsuperscript{*}}
\thanks{\textsuperscript{*}Corresponding author.}
\address[Chunlin Liu]{School of Mathematical Sciences, Dalian University of Technology, Dalian, 116024, P.R. China, and
Institute of Mathematics, Polish Academy of Sciences, ul. Śniadeckich 8, 00-656 Warszawa, Poland.}
\email{chunlinliu@mail.ustc.edu.cn}

\newcommand{\Id}{\textrm{Id}}
\newcommand{\ssq}{\subseteq}
\newcommand{\1}{\underline 1}
\renewcommand{\epsilon}{\varepsilon}
\newcommand{\eps}{\varepsilon}\renewcommand{\epsilon}{\varepsilon}
\renewcommand{\subset}{\subseteq}
\newcommand{\choice}{choice domain} %domain of independence??
\newcommand{\mc}{\mathcal}
\renewcommand{\:}{\colon}

\begin{document}
	\title[Idempotents in the Ellis semigroup of Floyd-Auslander systems
    ]{Idempotents in the Ellis semigroup of Floyd-Auslander systems}

	\begin{abstract}
  We study minimal idempotents $J^{\mathrm{min}}(X)$ in the Ellis semigroup $E(X)$ associated with a Floyd-Auslander system $(X,T)$.
We show that $(X,T)$ is non-tame if and only if $|J^{\mathrm{min}}(X)| > 2^{\aleph_0}$, which happens exactly when the factor map onto the maximal equicontinuous factor possesses uncountably many non-invertible fibres.

This yields an easy-to-check criterion for distinguishing tame from non-tame Floyd-Auslander systems and, more importantly, provides an entire family of regular almost automorphic systems with $|J^{\mathrm{min}}(X)| > 2^{\aleph_0}$.
Notably, all previously known regular almost automorphic non-tame systems exhibited only a small (i.e.\ $\leq 2^{\aleph_0}$) set of minimal idempotents.

We obtain our result by leveraging an alternative characterisation of (non)-tameness through, what we call, \emph{choice domains}.
\end{abstract}
 \maketitle
	%\baselineskip 15pt   % between lines
	%12pt-standard, 24pt-double space, 0.1in-narrow
	\parskip 10pt
       % between paragraphs
\section{Introduction}
Given a topological dynamical system (TDS), that is, a pair $(X,T)$ where $X$ is a compact metric space and $T\colon X \to X$ is continuous, an algebraic invariant that helps to understand its dynamics is the \emph{Ellis semigroup} $E(X,T)$, defined as the closure in $X^X$ of the set of iterates of $T$.
The Ellis semigroup is well known to encode key structural features of a topological dynamical system; for example, it characterises properties such as distality and minimality, and helps to identify the maximal equicontinuous factor\cite{Ellis1969,Auslanderbook,MR603625}.

However, one major obstacle in studying the semigroup $E(X,T)$ is the complex nature of the semigroup itself, as often, $E(X,T)$ is \emph{non-tame}: it contains a copy of the Stone-Czech compactification of $\N$.
By contrast, tame (i.e.\ not non-tame, equivalently, $|E(X,T)|\leq 2^{\aleph_0}$) systems form a rather restricted class: if $(X,T)$ is minimal (all orbits dense) and tame, then $(X,T)$ is \emph{regular almost automorphic}, that is, the factor map $\pi$  from $(X,T)$ to its maximal equicontinuous factor $(X_{\textrm{eq}},T_{\textrm{eq}})$ is such that $\{x\in X_{\textrm{eq}}\: |\pi^{-1}x|=1\}$ is of full measure (with respect to the unique invariant measure of $(X_{\textrm{eq}},T_{\textrm{eq}})$) \cite{Glasner2018,FuhrmannGlasnerJagerOertel2021}. 
Nonetheless, there has been ample progress in understanding $E(X,T)$ as well as the connection between its algebraic properties and the dynamics of $(X,T)$  for a growing number of systems, see e.g.\ \cite{HaddadJohnson1997,Huang2006,Donoso2014,HLSY2021,KellendonkReem2022,FuhrmannKellendonkYassawi2024,GlanserMegrelishvili,Kellendonk2024,LWX2025,LXZ2025} (and references therein) for a very incomplete list.

The contribution of the present article is not merely to extend this list by one more class of examples but rather, is three-fold.
First of all, by completing the analysis started in \cite{HaddadJohnson1997}, we want to expand our understanding of the Ellis-semigroup of $\Z$-actions beyond totally disconnected spaces: arguments pertaining to the tameness or non-tameness of a system often leverage some form of a shift-representation (possibly with an infinite alphabet, see e.g. \cite{FuhrmannDominik}); while such representations are, in principle, always available, in explicit examples---like the ones studied in this article---such abstract tools are of little help.
Second, our results show a previously unobserved phenomenon: we show that it is well possible among regular almost automorphic systems to not only have a big ($>2^{\aleph_0}$) Ellis semigroup but actually, already have a big collection of minimal idempotents in $E(X,T)$.
We derive this result by exploiting the existence of what we term \emph{uncountable choice domains} in the full shift.
Third, through introducing the notion of choice domains, we provide an explicit and readily verifiable criterion for detecting non-tameness. 
We expect this criterion to serve as a useful alternative to the well-known approach via infinite independence sets \cite{KerrLi2007}.

\paragraph{Main results}
Our primary contribution concerns an analysis of the idempotents in the Ellis semigroup of Floyd-Auslander systems.
Roughly speaking, Floyd-Auslander systems $(X,T)$ are constructed by repeatedly
subdividing the unit square into thin vertical rectangles, some of full height and some of half height.\footnote{We note that the naming convention for these systems is not entirely consistent. Haddad and Johnson use the term \emph{generalised Auslander systems}, referring to Auslander's 1959 example \cite{Auslander}, which was itself inspired by earlier work of Floyd \cite{Floyd1949}. In this article we adopt the name \emph{Floyd-Auslander systems}, which has been used in recent literature and acknowledges the contributions of both authors.}  Taking the intersection of all subdivision stages produces a
compact set $X \ssq [0,1]^2$ whose projection to the first coordinate is a
Cantor set $C$.  
For each $\alpha \in C$, the vertical fibre $\{\alpha\} \times [0,1] \cap X$ is either a single point or a vertical line segment of dyadic length, depending on the sequence of full-height versus half-height rectangles encountered in the construction.
The map $T$ then acts on the horizontal coordinate like an odometer; in the vertical direction, it acts by sending each vertical fibre that is a line segment linearly onto the corresponding line segment in the next fibre.  
See Section~\ref{sec: Floyd-Auslander systems} for a detailed and analytically convenient description of Floyd-Auslander systems, which we use in place of the more geometric presentation commonly found in the literature (and summarised above), see e.g.\ \cite{HaddadJohnson1997}.

To the authors’ knowledge, the set of idempotents in the Ellis semigroup of minimal Floyd-Auslander systems was first comprehensively studied in \cite{HaddadJohnson1997}. There, it was shown that if the set of (non-degenerate) interval fibres is countable, then the set of idempotents \(J(X) \subseteq E(X,T)\) has cardinality \(2^{\aleph_0}\). 
% More precisely, the non-trivial idempotents in \(E(X,T)\) are of the form
% \[
% f_y : X \to X, \qquad (\alpha, z) \mapsto (\alpha, b_\alpha + \ell(L_\alpha) y),
% \]
% where \(y \in [0,1]\), \(b_\alpha\) denotes the bottom point of the fibre \(L_\alpha = \{\alpha\} \times [0,1] \cap X\), and \(\ell(L_\alpha)\) is the length of the fibre \(L_\alpha\); any such \(f_y\) is realised as an idempotent.
We remark here that countability of the interval fibres is straightforward to verify: if $Q(n)$ denotes the number of full-height rectangles in the \(n\)-th stage of the construction, then there are countably many interval fibres if and only if 
\(\Lambda = \{ n \in \mathbb{N} : |Q(n)| \ge 2 \}\) is finite.

Besides this, the authors' of \cite{HaddadJohnson1997} observed that “in general \mbox{[Floyd-]}Auslander systems, $J(X)$ could be quite complex”.
It is this remaining complex case that we shed light on with the following result.

\begin{introthm}\label{thm: main 1 intro}
    Let $(X,T)$ be a minimal Floyd-Auslander system.
Then the following statements are equivalent.
    \begin{enumerate}
        \item $(X,T)$ is non-tame;\label{item: 1}
        \medskip
        \item the set of minimal idempotents $J^{\mathrm{min}}(X)$ is large: $|J^{\mathrm{min}}(X)|>2^{\aleph_0}$;\label{item: 4}
        \medskip
        \item $\L=\{n\in \N\colon |Q(n)|\geq 2\}$ is infinite.\label{item: 6}
    \end{enumerate}
\end{introthm}
We want to emphasize that the demanding part of proving the above statement is the implication \eqref{item: 6} $\Rightarrow$ \eqref{item: 4}.
In fact, \eqref{item: 1} $\Rightarrow$ \eqref{item: 6} is essentially shown in \cite{GlanserMegrelishvili}, while \eqref{item: 4} $\Rightarrow$ \eqref{item: 1} is trivial.

Our second main result is, in fact, a prerequisite for proving our first 
main result and technically far less involved.
It is nonetheless, interesting in its own right and deserves to be highlighted separately.
To formulate it, we need the following terminology.
  Given a binary subshift $(X,\sigma)$, we call $\mathcal T\ssq X$ \emph{a choice domain} if for all $m,n \in \N$, any subset of $n$ distinct points
$(x^1_\ell)_{\ell\in \N},\ldots,(x^n_\ell)_{\ell\in \N}\in \mathcal T$, and any collection of choice functions $(\phi^1_k)_{k=1}^m,\ldots,(\phi^n_k)_{k=1}^m\in \{0,1\}^m$  there are times $\tau_1<\ldots <\tau_m$ such that
\[
 x^i_{\tau_k}=\phi^i_k \qquad (i=1,\ldots,n, \ k=1,\ldots,m).
\]
 \begin{introthm}\label{thm: choice domains intro}
     A binary subshift is non-tame if and only if it admits an uncountable choice domain.
 \end{introthm}
We remark that the restriction to shifts is solely for convenience. 
The extension of the notion of choice domains, and of the preceding theorem, to general TDS is straightforward. 
However, in order to keep the notational overhead---already considerable in this article---to a minimum, we have decided to confine the discussion of choice domains to the setting directly relevant for proving Theorem~\ref{thm: main 1 intro}, which is precisely the setting of Theorem~\ref{thm: choice domains intro}.

\paragraph{Outline}
The next section collects the general background required throughout the article.
Section~\ref{sec: preparations} then introduces two key prerequisites: 
Theorem~\ref{thm: choice domains intro} and a technical sufficient criterion for the existence of many minimal idempotents.
The final section combines these ingredients to prove Theorem~\ref{thm: main 1 intro}.

\noindent
{\bf Acknowledgments.}
The first author would like to thank Johannes Kellendonk and Reem Yassawi for valuable discussions concerning choice domains and for highlighting the question of whether regular almost automorphic systems may possess many idempotents.
The second author gratefully acknowledges valuable discussions with Wen Huang. This work  was supported by   the
		Postdoctoral Fellowship Program and China Postdoctoral
		Science Foundation under Grant Number BX20250067, and the China Postdoctoral Science
		Foundation under Grant Number 2025M773074.

\section{Preliminaries and background}
We begin by briefly reviewing some basic facts concerning the Ellis semigroup of a TDS, and then discuss the construction of Floyd-Auslander systems. While the former is largely standard, our construction of Floyd-Auslander systems departs from the usual approach, making Section~\ref{sec: Floyd-Auslander systems} essential for the remainder of the article.

\subsection{The Ellis semigroup, its idempotents, and tameness}\label{sec: prelims tameness elis semigp}

Given a TDS $(X,T)$, its \emph{Ellis semigroup} $E(X,T)$ is defined as
\[
E(X,T) = E(X) := \overline{\{T^n : n \in \mathbb{N}\}} \subseteq X^X,
\]
where the closure is taken in the product topology and the semigroup operation is given by composition.
\begin{rem}\label{rem: forward Ellis gp}
    Note that strictly speaking, if $T$ is bijective, the above only defines its \emph{forward} Ellis semigroup (with the \emph{full} Ellis semigroup given by the closure over all $T^n$ with $n\in \Z$).
The distinction is not particularly relevant for us as the idempotents which we construct already lie in $E(X)$ (as defined above).
\end{rem}

It is straightforward to verify that $E(X)$ is a compact right topological semigroup.
Recall that a semigroup $E$ equipped with a topology is said to be \emph{right topological} if, for each $g \in E$, the map $E \ni f \mapsto f g$ is continuous.

Given a semigroup $E$, an element $f \in E$ is called an \emph{idempotent} if $f^2 = f$.
The following classical result is due to Naka\-mura~\cite{MR48467} (see also \cite[Lemma~8.4]{MR603625}).

\begin{lem}\label{lem:existence of idempotent}
If $E$ is a compact right topological semigroup, then $E$ contains an idempotent.
\end{lem}

We denote by $J(X)$ the set of all idempotents of $E(X)$.
The above statement gives $J(X)\neq \emptyset$.
There is a natural partial order on $J(X)$ defined by
\[
e \leq f \quad \text{if and only if} \quad e = ef \qquad (e,f \in J(X)).
\]
An element $f \in J(X)$ is called a \emph{minimal idempotent} if it is minimal with respect to $\leq$ (see \cite{HindmanStrauss2012} for details).
We denote by $J^{\mathrm{min}}(X)$ the set of all minimal idempotents of $E(X)$.
Since $E(X)$ is a compact right topological semigroup, it follows that $J^{\mathrm{min}}(X) \neq \emptyset$ (see, for example, \cite[Corollary~2.6 and Theorem~2.9]{HindmanStrauss2012}).

The Ellis semigroup is a fundamental, easily defined, and yet remarkably rich conjugacy invariant of topological dynamical systems.
However, gaining a concrete understanding of $E(X)$ for specific systems is notoriously difficult; see, for instance, \cite{Glasner1993,Donoso2014,Staynova2021,KellendonkReem2022}.
One reason for this difficulty is that many systems are \emph{non-tame}, meaning that, topologically, their Ellis semigroup is as large as possible.
More precisely, a system $(X,T)$ is called \emph{non-tame} if its Ellis semigroup contains a copy of the Stone–Čech compactification $\beta \mathbb{N}$; otherwise, $(X,T)$ is called \emph{tame}.

The study of (non-)tameness has been motivated by several factors, one of the most prominent being that a thorough understanding of these notions is crucial for developing a comprehensive algebraic theory of topological dynamics.
During the past two decades, significant progress has been achieved, much of it relying on the fact that there are several equivalent characterisations of \mbox{(non-)}tameness.
Before stating two such characterisations that are particularly relevant for the present work, we introduce the following concept.

Given a TDS $(X,T)$ and subsets $A_0, A_1 \subseteq X$, we say that
$J \subseteq \mathbb{N}$ is an \emph{independence set} for $(A_0, A_1)$ if,
for every finite subset $I \subseteq J$ and each
$\varphi \in \{0,1\}^I$, there exists $x \in X$ such that
$T^i(x) \in A_{\varphi(i)}$ for each $i \in I$ \cite{KerrLi2007}. 

For the following statement, see, for example,
\cite{GlasnerMegrelishvili2006,Huang2006,KerrLi2007}.

\begin{thm}\label{thm: characterisations of tameness}
Let $(X,T)$ be a TDS. Then the following are equivalent.
\begin{enumerate}
  \item $(X,T)$ is non-tame;
  \item $|E(X)| > 2^{\aleph_0}$;
  \item There exist closed, disjoint subsets $A_0, A_1 \subseteq X$ that admit an infinite independence set.
\end{enumerate}
\end{thm}

\begin{rem}
    Similar to Remark~\ref{rem: forward Ellis gp}, we would like to mention that the above statement actually describes 
    \emph{forward} non-tameness; given an invertible system, ``general" non-tameness, i.e., the situation where the full Ellis semigroup has cardinality bigger than $ 2^{\aleph_0}$, is already equivalent to the existence of an infinite independence set in $\Z$. 
    See \cite{FuhrmannKellendonkYassawi2024} for an example of a non-tame system where independence sets can contain at most finitely many positive integers.
\end{rem}

In case of binary subshifts $(X,\sigma)$---that is, $X\ssq \{0,1\}^\N$ and $\sigma$ denotes the left-shift on $X$---item (3) of Theorem~\ref{thm: characterisations of tameness} becomes particularly nice (see \cite[Remark~1.6]{FuhrmannKellendonkYassawi2024} for the details).

\begin{cor}\label{cor: non-tame subshifts}
    Suppose $(X,\sigma)$ is a binary subshift.
    Then $(X,\sigma)$ is non-tame if and only if
    $(A_0,A_1)$ admits an infinite independence set, where $A_0=\{(x_n)\in X\colon x_1=0\}$ denotes the $0$-cylinder and $A_1=\{(x_n)\in X\colon x_1=1\}$ the $1$-cylinder.
\end{cor}

\subsection{Floyd-Auslander systems as skew-products}\label{sec: Floyd-Auslander systems}
     In the remainder of this article, we work with a description
     of Floyd-Auslander systems as skew-products which is particularly convenient for our analysis.
      Given a sequence $(p_n)_{n\in \N}\ssq \N_{>1}$, we write $(\Sigma_{(p_n)},\1)$ for the odometer with $\Sigma_{(p_n)}=\prod_{n\in\N}\{0,1,\ldots,p_{n}-1\}$ and $\1=10^\infty\in \Sigma_{(p_n)}$.

For each $n\in\mathbb{N}$ and $j\in[p_n]:=\{0,1,\ldots,p_n-1\}$, let $\l_{j}^n$ be a function belonging to $\{\l_0,\l_1,\l_2\}$, where $\l_0,\l_1$ and $\l_2$ are the following self-maps on the unit interval $[0,1]$
       \[\l_0:x\mapsto\frac{1}{2}x,\quad  \l_1:x\mapsto x,\quad  \l_2:x\mapsto\frac{1}{2}x+\frac{1}{2}.\]
       For $n\in\mathbb{N}$,  set
       \[Q(n):=\{j\in[p_n]:\l^n_j=\l_1\},\quad H_0(n):=\{j\in[p_n]:\l^n_j=\l_0\},\quad H_2(n):=\{j\in[p_n]:\l^n_j=\l_2\}.\]

Consider
       \begin{align*}X&=\{(\a,y)\in\Sigma_{(p_n)}\times [0,1]: \text{ }y=\lim_{n\to\infty}\l_{\a}^{n}(z)\text{ for some }z\in[0,1]\}\\
       &=\bigcap_{n\geq 1}\{(\a,y)\in\Sigma_{(p_n)}\times [0,1]: \text{ }y=\l_{\a}^{n}(z)\text{ for some }z\in[0,1]\},\end{align*}
       where $\l_{\a}^{n}:=\l^1_{\a_1}\circ\ldots\circ\l^n_{\a_n}$.
       Note that $X$ is a compact subset of $\Sigma_{(p_n)}\times [0,1]$
       such that for each $\alpha$, the fibre $L_\alpha:=\{(\alpha,y)\in X\}$ is a singleton or a non-degenerate closed line segment.
      We call a fibre $L_\alpha$ \emph{maximal} if $|\pi_2(L_\a)|\geq |\pi_2(L_\beta)|$ for all $\beta\in \odo$, where $\pi_2$ denotes the projection to the second coordinate.

       Define $T:X\to X$ by
       \[T(\a,y)=(\a+\underline{1},\lim_{n\to\infty}\l_{\a+\underline{1}}^n(z)),\]
       where $z$ is such that $y=\lim_{n\to \infty} \lambda^n_{\alpha}(z)$.

       \begin{defn}\label{defn: Floyd-Auslander system}
           We call $(X,T)$ a Floyd-Auslander system if the following two assumptions are met.
       \begin{itemize}
       	\item There are infinitely many $n$ with $\l_0^n\in\{\l_0,\l_2\}$ and infinitely many $m$ with $\l_{p_m-1}^m\in\{\l_0,\l_2\}$.
\medskip
       	\item There are at most finitely many $n$ such that $Q(n)=\emptyset$.
       \end{itemize}
       \end{defn}     
       The first of the above points implies that  $T$ is continuous, that is, $(X,T)$ is indeed a TDS.
       The second point ensures that $(X,T)$ is not just conjugate to the odometer in the first coordinate.
       Note that if, in addition, $|Q(n)|\geq 2$ for infinitely many $n$, then there are uncountably many maximal fibres.
       Readers may convince themselves that the class of Floyd-Auslander systems as defined above is indeed---up to conjugacy---the one discussed in \cite{HaddadJohnson1997}.

The following theorem characterizes minimality for Floyd--Auslander systems.
See \cite{Auslander,Auslanderbook,HaddadJohnson1997} for a proof.
       \begin{thm}[{cf. \cite[Theorem 3.4]{HaddadJohnson1997}}]\label{thm:FA minimal}
       	A Floyd-Auslander system is minimal if and only if there exist infinitely many $n$ with $|H_0(n)|>0$ and infinitely many $m$ with $|H_2(m)|>0$.
       \end{thm}

\begin{rem}\label{rem: classical Floyd system}
Auslander's variation \cite{Auslander} of Floyd's original example \cite{Floyd1949} is obtained for
$p_n=3$ and $\lambda^n_j=\lambda_j$ for each $n\in\N$ and $j=0,1,2$.
\end{rem}

\noindent
{\textbf{Standing assumption.}}
For convenience, we impose the standing assumption that $1\le |Q(n)|<p_n$ for all $n\in\N$.
In other words, we throughout assume that
\[H_0(n)\cup H_2(n)\neq\emptyset,\text{ for all }n\in\N.\]

Note that for our purposes, this results in no loss of generality: every Floyd-Auslander system $(X,T)$ is conjugate to a Floyd-Auslander system $(\tilde X,\tilde T)$ satisfying the above assumption (see below) and a conjugacy between $(X,T)$ and $(\tilde X,\tilde T)$ naturally induces a semigroup isomorphism between $E(X)$ and $E(\tilde X)$ \cite{Auslanderbook}; consequently, the cardinality of $J^{\textrm{min}}(X)$ equals that of $J^{\textrm{min}}(\tilde X)$.

To see that indeed, every Floyd-Auslander system is conjugate to one satisfying the above, note first that for any fixed $N\in \N$, the Floyd-Auslander system $(X,T)$---over an odometer $(\Sigma_{(p_n)},\underline{1})$ and defined through the collection of maps $\lambda_j^n$ ($n\in\N$, $j\in[p_n]$)---is conjugate to the Floyd-Auslander $(\tilde X,\tilde T)$ over the same odometer but with
\[
{\tilde \lambda}_j^n=\lambda_1 \text{ for } n=1,\ldots,N \text{ and } j\in [p_n] \qquad \text{ and } \qquad {\tilde \lambda}_j^n=\lambda_j^n \text{ otherwise.}
\]
Indeed, a conjugacy $h\:(\tilde X,\tilde T)\to(X,T)$ is given by
\[
h\: (\alpha,\lim_{n\to \infty} {\tilde \lambda}_\alpha^n(z))\mapsto (\alpha,\lim_{n\to \infty} {\lambda}_\alpha^n(z))
\]
for each $z\in[0,1]$.
This implies that every Floyd-Auslander system---where, by definition, there is $N\in\N$ such that $Q(n)\neq \emptyset$ for all $n\geq N$---is conjugate to one with $|Q(n)|\geq 1$ for \emph{all} $n$.

Now, let us assume that $(X,T)$ is such that $1 \leq |Q(n)| \leq p_n$ but with infinite (the finite case can be dealt with similarly) strictly increasing sequences $(n_\ell)$ and $(N_\ell)$ in $\N$ where $n_\ell<N_\ell<n_{\ell+1}$  and $|Q(n_\ell+k)|=p_{n_\ell+k}$ for $0\leq k<N_\ell-n_\ell$ while $|Q(N_\ell)|<p_{N_\ell}$.
Set $\tilde p_{N_\ell}= p_{n_\ell}\cdot p_{n_\ell+1}\cdots p_{N_\ell}$ for each $\ell$
and $\tilde p_n=p_n$ whenever $N_\ell<n<n_{\ell+1}$ for some $\ell$---that is, whenever $n$ is not between $n_\ell$ and $N_\ell$ for any $\ell$.
Let $\mc N=\{n\in\N\: N_\ell\leq n<n_{\ell+1}\text{ for some }\ell\in \N\}$
and put $\tilde \Sigma=\prod_{n\in \mc N}[\tilde p_n]$.
Then the odometer $(\tilde \Sigma,\underline{1})$ is conjugate to $(\Sigma_{(p_n)},\underline{1})$ through the canonical bijection between $\tilde \Sigma$ and $\Sigma$.
Further, if we set
\[\tilde \lambda_j^{N_\ell}=\lambda^{n_\ell}_{j_{n_\ell}}\circ \lambda^{n_\ell+1}_{j_{n_\ell}+1}\circ \ldots \circ \lambda^{N_\ell}_{j_{N_\ell}}=\lambda^{N_\ell}_{j_{N_\ell}} \quad (\ell \in \N, \ j\in [\tilde p_{N_\ell}]),
\]
 where $j=j_{n_\ell}+j_{n_\ell+1}\cdot p_{n_\ell}+\ldots+j_{N_\ell}\cdot p_{n_\ell}\cdot p_{n_\ell+1}\cdots p_{N_\ell-1}$, and 
 \[
 \tilde \lambda_j^{n}=\lambda_j^{n} \text{ for all } n\in \mc N\setminus \{N_\ell\: \ell\in \N\},
 \]
 then we obtain that the associated Floyd-Auslander system $(\tilde X,\tilde T)$ is conjugate
 to $(X,T)$ and satisfies $1\leq|Q(n)|<p_n$ for all $n$.

\section{Preparations}\label{sec: preparations}
In this section we develop tools that will be used repeatedly in the proof of Theorem~\ref{thm: main 1 intro}: a convenient characterisation of (non-)tameness and a technical condition ensuring the existence of many minimal idempotents for Floyd-Auslander systems.

\subsection{An alternative characterisation of tameness}\label{sec: choice domains}
In the following, we establish an alternative characterisation of (non-)tameness.
Towards proving our main result Theorem~\ref{thm: main 1 intro}, we only use some part of the following discussion (specifically, Lemma~\ref{lem: uncountable choice domains}).
However, as we believe that our alternative view on tameness is of independent interest, we collect a few basic facts here which go slightly beyond the bare minimum of what is actually needed to prove Theorem~\ref{thm: main 1 intro}.

For notational clarity and brevity, we restrict our discussion to binary subshifts $(X,\sigma)$ (i.e.\ $X\ssq \{0,1\}^\N$) and, as in Section~\ref{sec: prelims tameness elis semigp} and without any further mentioning, \emph{forward} tameness.
The generalisation to general TDS and the translation to ``forward-and-backward" tameness is immediate.
Let us also remark that although our main result does not actually deal with subshifts, this is exactly the setup we need for its proof, see Section~\ref{sec: many idempotents}.

\begin{defn}
    Given a binary subshift $(X,\sigma)$, we call $\mathcal T\ssq X$ \emph{a choice domain} if for all $m,n \in \N$, any subset of $n$ distinct points
$(x^1_\ell)_{\ell\in \N},\ldots,(x^n_\ell)_{\ell\in \N}\in \mathcal T$, and any collection of choice functions $(\phi^1_k)_{k=1}^m,\ldots,(\phi^n_k)_{k=1}^m\in \{0,1\}^m$  there are times $m<\tau_1<\ldots <\tau_m$ such that
\[
 x^i_{\tau_k}=\phi^i_k \qquad (i=1,\ldots,n, \ k=1,\ldots,m).
\]
We call such times $m<\tau_1<\ldots <\tau_m$  \emph{$m$-realising} for $(x^1_\ell),\ldots, (x^n_\ell)$ and $(\phi^1_k),\ldots,(\phi^n_k)$.
We further call the choice domain  $\mathcal T$ \emph{maximal} if whenever $\mathcal T'\ssq X$ is a choice domain and $\mathcal T\ssq \mathcal T'$, then $\mathcal T=\mathcal T'$.
\end{defn}

\begin{rem}
Note that this definition differs slightly from the introduction by imposing the assumption $\tau_1 > m$. 
We introduce this additional constraint here solely for technical convenience; the fact that $m$ can be chosen arbitrarily large implies that both definitions are indeed equivalent.
\end{rem}
\begin{rem}
    Note that being a choice domain is a property of finite character so that Zorn's Lemma indeed ensures the existence of a maximal element.
    Note further that the existence of non-empty choice domains (and hence, non-emptiness of maximal choice domains) is equivalent to the existence of a sequence $(x_n)\in X$ which is not eventually constant.
\end{rem}

\begin{prop}\label{prop: alternative characterisation of choice domains}
 Consider a binary subshift $(X,\sigma)$ and $\mathcal T\ssq X$.
 The following are equivalent.
 \begin{enumerate}
     \item[(a)] $\mathcal T$ is a choice domain.
     \medskip
     \item[(b)] For each finite subset of distinct points
$(x^1_\ell)_{\ell\in \N},\ldots,(x^n_\ell)_{\ell\in \N}\in \mathcal T$, each $m\in\N$ and $(\phi^i)_{i=1}^n\in\{0,1\}^n$, there is $\tau >m$ with
\[
 x^i_{\tau}=\phi^i.
\]

\item[(c)] For each finite subset of distinct points
$(x^1_\ell)_{\ell\in \N},\ldots,(x^n_\ell)_{\ell\in \N}\in \mathcal T$, and each
collection of choice functions 
$(\phi^1_k)_{k\in \N},\ldots,(\phi^n_k)_{k\in \N}\in \{0,1\}^\N$, there is a strictly increasing sequence 
$(\tau_k)_{k\in \N}$ in $\N$ such that
\[
 x^i_{\tau_k}=\phi^i_k \qquad (i=1,\ldots,n, \ k\in\N).
\]
 \end{enumerate}
\end{prop}
\begin{proof}
We only show how (b) implies (c); the other parts are even more immediate.
Consider $(x^1_\ell)_{\ell\in \N},\ldots,(x^n_\ell)_{\ell\in \N}\in \mathcal T$ and
$(\phi^1_k)_{k\in \N},\ldots,(\phi^n_k)_{k\in \N}\in \{0,1\}^\N$ as in (c).
By (b), there is $\tau_1 \in \mathbb{N}$ with 
\[
x^i_{\tau_1} = \phi^i_1 \qquad (i = 1, \ldots, n).
\]
Analogously, there is $\tau_2 > \tau_1$ with 
\[
x^i_{\tau_2} = \phi^i_2 \qquad (i = 1, \ldots, n).
\]
Continuing like this, we obtain $(\tau_k)_{k \in \mathbb{N}}$ as in (c).
\end{proof}

\begin{lem}\label{lem: uncountable choice domains}
    If $\mathcal T$ is a maximal choice domain of the full shift 
    $(\{0,1\}^\N,\sigma)$, then $\mathcal T$ is uncountable.
\end{lem}
\begin{proof}
Assume, for a contradiction, that $\mathcal{T} = \{x^1, x^2, \ldots\}$ is an enumeration of a maximal choice domain. 
Without loss of generality, we may assume that $\mathcal{T}$ is infinite; the finite case can be handled analogously.
By a diagonal argument, we can recursively construct a sequence $y \in \{0,1\}^{\mathbb{N}}$ such that $y \notin \mathcal{T}$ but $\mathcal{T} \cup \{y\}$ is still a choice domain, thereby contradicting the maximality of $\mathcal{T}$.

We provide the details for the reader's convenience.
Suppose we have already carried out the first $n$ stages of the construction for some $n \geq 1$ in a way that the first $N$ entries $y_1, \ldots, y_N$  (for some $N \in \mathbb{N}$) of $y$ have been specified.
Since $\mathcal{T}$ is a choice domain, for each $\phi = (\phi^i)_{i=1}^{n+1} \in \{0,1\}^{n+1}$ we can find a time $\tau_\phi > N$ such that
\[
x^i_{\tau_\phi} = \phi^i \qquad (i = 1, \ldots, n+1).
\]
Similarly, again for each $\phi = (\phi^i)_{i=1}^{n+1} \in \{0,1\}^{n+1}$ we can find $\tau_\phi' > \max_{\psi \in \{0,1\}^{n+1}} \tau_\psi$ satisfying
\[
x^i_{\tau_\phi'} = \phi^i \qquad (i = 1, \ldots, n+1).
\]
In the $n+1$-th stage of the construction, we now specify $y_\ell$ for $\ell \in \{N+1, \ldots, \max_{\phi \in \{0,1\}^{n+1}} \tau_\phi'\}$
as
\[
y_\ell =
\begin{cases}
1, & \text{if } \ell = \tau_\phi' \text{ for some } \phi \in \{0,1\}^{n+1},\\[2mm]
0, & \text{otherwise.}
\end{cases}
\]

\smallskip
\noindent Proceeding recursively, we obtain a sequence $y \in \{0,1\}^{\mathbb{N}}$ with $y \notin \mathcal T$ (since at each stage of the construction, $0=y_{\tau_{\phi}}\neq y_{\tau_{\phi}'}=1$) and such that $\mathcal{T} \cup \{y\}$ still satisfies condition~(b) of Proposition~\ref{prop: alternative characterisation of choice domains}.
The statement follows.
\end{proof}

Note that if $(X,\sigma)$ is non-tame, then Corollary~\ref{cor: non-tame subshifts} yields the existence of a strictly increasing sequence $(t_n)$ in $\N$ along which the binary full-shift is embedded in $X$.
More precisely, for each $\phi\in \{0,1\}^\N$, there is
$x^\phi\in X$ with $x^\phi_{t_n}=\phi_n$.
As a consequence, any choice domain $\mathcal T$ of the fullshift injects into a choice domain $\mathcal T_X$ in $X$ given by $\mathcal T_X=\{x^\phi\in X\: \phi\in \mathcal T\}$.
We hence obtain the following corollary to the previous lemma.
\begin{cor}
    Suppose $(X,\sigma)$ is a binary subshift.
If $(X,\sigma)$ non-tame, then it allows for an uncountable choice domain.
\end{cor}
In fact, the converse holds as well, which completes the proof of Theorem~\ref{thm: choice domains intro}.
\begin{proof}[{Proof of Theorem~\ref{thm: choice domains intro}}]
It remains to show that given a binary subshift $(X,\sigma)$,
the existence of an uncountable choice domain $\mc T$ implies non-tameness.
We do this by showing that $E(X)>2^{\aleph_0}$, see Theorem~\ref{thm: characterisations of tameness}.
Specifically,
we show that for each function $f\in \{0,1\}^\mathcal T$, there is
$\eta \in E(X)$ such that 
\begin{align}\label{eq: realising functions over choice domain}
    \eta(x)\in A_{f(x)} \quad(x\in \mathcal T),
\end{align}
where $A_i=\{(y_\ell)\in X\:y_1=i\}$ for $i=0,1$.

To that end, consider $[\mathcal T]^{<\omega}$---the collection of finite subsets of $\mathcal T$---equipped with the direction given by inclusion.
Given $f\in \{0,1\}^\mathcal T$, we define a net $\{t_{\{x^1,\ldots,x^n\}}\}_{\{x^1,\ldots,x^n\}\in [\mathcal T]^{<\omega}}$ in $\N$ by picking $\tau=t_{\{x^1,\ldots,x^n\}}$ such that
\begin{align*}
x^i_\tau = f(x^i), \quad \text{in other words, such that} \quad 
\sigma^\tau x^i \in A_{f(x^i)} \qquad (i=1,\ldots,n).
\end{align*}
Then, any accumulation point $\eta$ of $\{\sigma^{t_{\{x^1,\ldots,x^n\}}}\}_{\{x^1,\ldots,x^n\}\in [\mathcal T]^{<\omega}}$ satisfies \eqref{eq: realising functions over choice domain} and
the statement follows.
\end{proof}

\subsection{A sufficient criterion for many minimal idempotents in Floyd-Auslander systems}
Let $(X,T)$ be a Floyd--Auslander system.
In this subsection, we present a condition which ensures that the set of minimal idempotents $J^{\mathrm{min}}(X)$ has cardinality strictly larger than $2^{\aleph_0}$.
In Section~\ref{sec: many idempotents}, we will then see that this sufficient condition is satisfied by all minimal Floyd-Auslander systems with uncountably many interval fibres.

% Recall that throughout, we assume that $1\le |Q(n)|< p_n$ for all $n\in\N$.
We begin with some terminology.
 An uncountable subset $\mathcal{A}\subset \Sigma_{(p_n)}$ such that $L_{\a}$ is maximal for every $\a\in\mathcal{A}$ is called \emph{\good}\ if there exist strictly increasing sequences $\{q_\ell\}_{\ell\in\N}$, $\{k_\ell\}_{\ell\in\N}$ and $\{k_\ell'\}_{\ell\in\N}$ with the following properties.
\begin{itemize}
\item[(P1)] For every $\ell\in\N$ there exist $\xi_{q_\ell}^1,\xi_{q_\ell}^2\in Q(q_\ell)$ with
\[
0\le \xi_{q_\ell}^1<\xi_{q_\ell}^2\le p_{q_\ell}-1
\]
and such that for every $\a\in\mathcal{A}$,
\[
\a_{q_\ell}=\xi_{q_\ell}^1\quad (\ell\in\N).
\]

\item[(P2)] 
$H_{2}(k_\ell)\neq\emptyset$ for all $\ell\in\N$, and there exists $\zeta_{k_\ell}\in Q(k_\ell)$ such that for every $\a\in\mathcal{A}$,
\[
\a_{k_\ell}=\zeta_{k_\ell}\quad (\ell\in\N);
\]
 $H_{0}(k_\ell')\neq\emptyset$ for all $\ell\in\N$, and there exists $\zeta_{k_\ell'}\in Q(k'_\ell)$ such that for every $\a\in\mathcal{A}$,
\[
\a_{k'_\ell}=\zeta_{k'_\ell}\quad (\ell\in\N).
\]
\end{itemize}
We now say a Floyd-Auslander system $(X,T)$ satisfies property ($\ast$) if one of the following holds.

\vspace{-15pt}
\noindent\begin{minipage}[t]{0.92\textwidth}
\phantom{There exist  \( a\neq b \in [0,1] \)   and an uncountable subset $\mathcal{A} \subset \Sigma_{(p_n)}$, for which \( L_\alpha \) has}\tikzmark{start}%
\begin{enumerate}
    \item There exists $a\in(0,1)$ and a \good\ subset $\mathcal{A} \subset \Sigma_{(p_n)}$  such that for any $B \subset \mathcal{A}$, there is \( f_B \in J(X) \) and  a partition $B=B_1 \sqcup B_2$ by some (possibly empty) subsets $B_1,B_2\ssq B$ such that
\begin{itemize}
\item $f_B(L_\a)=(\a,b)$ if $\a \in \mathcal A \setminus B$ and where $b=b(\a)\in [0,a]$;
\item $f_B(L_\a)=(\a,c)$ if $\a \in B_1$ and where $c=c(\a)>a$;
    \item $\left.f_B\right|_{L_\a}=\left.\Id\right|_{L_\a}$ if $\a \in B_2$.
\end{itemize}
\item There exists $a\in(0,1)$ and a \good\ subset $\mathcal{A} \subset \Sigma_{(p_n)}$ such that for any $B \subset \mathcal{A}$, there is \( f_B \in J(X) \) and  a partition $B=B_1 \sqcup B_2$ by some (possibly empty) subsets $B_1,B_2\ssq B$ such that
\begin{itemize}
\item $f_B(L_\a)=(\a,b)$ if $\a \in \mathcal A \setminus B$ and where $b=b(\a)\in [a,1]$;
\item $f_B(L_\a)=(\a,c)$ if $\a \in B_1$ and where $c=c(\a)<a$;
    \item $\left.f_B\right|_{L_\a}=\left.\Id\right|_{L_\a}$ if $\a \in B_2$.
\end{itemize}
\end{enumerate}

\phantom{There exist  \( a\neq b \in [0,1] \)   and an uncountable subset $\mathcal{A} \subset \Sigma_{(p_n)}$, for which \( L_\alpha \) has}\tikzmark{end}
\end{minipage}
% Draw the brace and label

\tikz[remember picture, overlay]{
  \draw[decorate, decoration={brace, amplitude=7pt}]
    ($(start.north east)+(1.9,-0.2)$) -- ($(end.south east)+(1.9,0.5)$)
    node[midway,right=5pt] {\large $(\ast)$};
}

\vspace{-30pt}
Note that here and at various other places in this work,
given an idempotent $f$ that collapses a fibre $L_\a$ (that is, $|f(L_\a)|=1$), we identify the singleton image $f(L_\a)$ with its unique element.

In Section~\ref{sec: many idempotents}, we will prove that if $\L:=\{n\in\N:|Q(n)|\ge2 \}$ is infinite, then the Floyd-Auslander system $(X,T)$ satisfies property ($\ast$).
In the remainder of this paragraph, we show that property ($\ast$) implies $|J^{\mathrm{min}}(X)|>2^{\aleph_0}$.
To that end, observe the following basic facts whose straightforward proofs are left to the reader.
\begin{fact}\label{fact1}
Let $(X,T)$ be a Floyd--Auslander system. For any \( f \in J(X) \) and \( \alpha \in \Sigma_{(p_n)} \), we have
\[
\left.f\right|_{L_\alpha} = \left. \Id\right|_{L_\alpha} \quad \text{or} \quad f(L_\alpha) = (\alpha, z) \text{ for some } z \in L_\alpha.
\]
\end{fact}

\begin{fact}\label{fact2}
Let $(X,T)$ be a Floyd--Auslander system. For any \( \alpha \in \Sigma_{(p_n)} \), there exists \( f \in J(X) \) with
\[
f(L_\alpha) = (\alpha, z) \text{ for some } z \in L_\alpha.
\]
\end{fact}
The next lemma provides a characterization of minimal idempotents in Floyd-Auslander systems.
\begin{lem}\label{lem:when is minimal}
    Let $(X,T)$ be a Floyd-Auslander system and consider $f\in E(X)$ with $\pi_1\circ f(\alpha,y)=\alpha$ for all $(\a,y)\in X$.
Then $f\in J^{\mathrm{min}}(X)$ if and only if
    $f$ collapses all fibres.
\end{lem}
\begin{proof}
 The ``if"-direction is straightforward and the ``only if"-part follows from Fact~\ref{fact1} and Fact~\ref{fact2}.
\end{proof}

\begin{lem}\label{lem:existence of minimal}
Let $(X,T)$ be a Floyd--Auslander system and suppose $\mc A\ssq \Sigma_{(p_n)}$ is \good.
Fix $a\in(0,1)$.
There exist $f,g\in J^{\mathrm{min}}(X)$ such that for every $\a\in\mc A$
\[
(\{\a\}\times[0,a])\cap f(L_\alpha)=\emptyset \quad \text{and} \quad
 (\{\a\}\times[a,1])\cap g(L_\alpha)=\emptyset.
\]
\end{lem}
\begin{proof} We only consider the first case (involving $f$), the second one works similarly (using the sequence $\{k_\ell'\}$ instead of $\{k_\ell\}$ in the below argument).

Fix $j\in\N$ such that $\l_2^j(0)>a$.
Pick $g$ from $J^{\mathrm{min}}(X)\ssq E(X)$ and a net $\{t_\xi\}_{\xi\in I}\ssq \N$ of the form
\[t_\xi=\overbrace{00\ldots00}^{\ell_\xi}\xi_1\ldots\xi_{n_\xi}\]
such that $\lim_{\xi\in I}T^{t_\xi}=g$ and $\lim_{\xi\in I}t_\xi=0$ (in $\odo$). Since $\mc{A}$ is \good,
possibly going over to a subnet, we may assume without loss of generality that for each $\xi\in I$, there exist $k_{\xi}^1<K_{\xi}^1<k_{\xi}^2<K_{\xi}^2<\ldots<k_{\xi}^j<K_{\xi}^j< \ell_\xi$ with $\lim_{\xi\in I}k_{\xi}^1=\infty$ and such that
 $k_\xi^i\in \{k_\ell\}_{\ell\in\N}$ and  $K_\xi^i\in \{q_\ell\}_{\ell\in\N}$ for $i=1,\ldots,j$ (with $\{q_\ell\}$ and $\{k_\ell\}$ as in (P1) and (P2), respectively). 
 Hence, there exist $\Delta_\xi^{i}\in [p_{k^i_\xi}]$ and $\Gamma_\xi^{i}\in [p_{K^i_\xi}]$ such that for all $\a\in\mc A$,
$\a_{k^i_\xi}+\Delta_\xi^i\in H_2(k^i_\xi) \pmod {p_{k_\xi^i}}$ and $\a_{K^i_\xi}+\Gamma_\xi^i\in Q(K^i_\xi)$  for $i=1,\ldots,j$.
Therefore,
for $\xi \in I$ and $i=1,\ldots,j$, we may choose  $u_1^i\in [p_{k_\xi^i+1}], \ldots, u_{K_\xi^i-k_\xi^i}^i\in [p_{K_\xi^i}]$
such that
\[\tau_{\xi}:=\underbrace{0\ldots 0\Delta_\xi^1}_{k_\xi^1}\underbrace{u_1^1\ldots u_{K_\xi^1-k_\xi^1}^1 0\ldots 0
\Delta_\xi^2}_{k_\xi^2-k_\xi^1}\underbrace{u_1^2\ldots u_{K_\xi^2-k_\xi^2}^2 0\ldots 0\Delta_\xi^j}_{k_\xi^2-k_\xi^j}u_1^j\ldots u_{K_\xi^j-k_\xi^j}^j0\ldots 0\]
satisfies
\[
(\a+\tau_\xi)_n \in
\begin{cases}
H_2(n) & \text{ for } n=k_\xi^1,\ldots, k_\xi^j,\\
Q(n) & \text{ otherwise,}
\end{cases}
\]
for each $\alpha \in \mc A$.
We define a new net $\{t_\xi'\}_{\xi\in I}$ by $t_\xi'=\tau_\xi+t_\xi$.
Note that any accumulation point of the net \( \{T^{t_\xi'}\}_{\xi \in I}\ssq E(X) \) has the desired properties.
\end{proof}

\begin{prop}\label{prop: star implies main}
      Let $(X,T)$ be a Floyd-Auslander system with property  $(\ast)$.
Then $|J^{\mathrm{min}}(X)|>2^{\aleph_0}$.
\end{prop}
\begin{proof}
We only consider case (1) in property ($\ast$) as case (2) is similar (with the idempotent $g$ from Lemma~\ref{lem:existence of minimal} instead of $f$ in the following argument).
Consider $a\in (0,1)$ and $\mathcal A\ssq \odo$ as in ($\ast$).
    By Lemma~\ref{lem:existence of minimal}, there is a minimal  idempotent $f$ such that  for every \( \alpha \in \mathcal{A} \), we have \( (\{\a\}\times [0,a]) \cap f(L_\alpha)=\emptyset \).
Due to Lemma~\ref{lem:when is minimal}, given $B\subset\mathcal{A}$ and $f_B$ as in ($\ast$), $f_B\circ f$ is a minimal idempotent.
Therefore, the statement follows by proving $f_{B}\circ f\neq f_{B'}\circ f$ for any pair $B\neq B'$ of subsets of $\mathcal{A}$.
   To that end, choose \( \alpha \in B' \setminus B \) (assuming without loss of generality that $B'\setminus B\neq \emptyset$).
Then
\[
f_{B} \circ f(L_\alpha) = f_{B}(L_\alpha)\in \{\a\}\times [0,a] \qquad
\text{ and }
\qquad
f_{B'} \circ f(L_\alpha)\in\{(\a,c)\colon c>a\}.
\]
In particular,
$f_{B'} \circ f(L_\alpha) \notin \{\a\}\times [0,a]$ so that $f_{B} \circ f \neq f_{B'} \circ f$.
This completes the proof.
\end{proof}

\section{Non-tame implies many minimal idempotents}\label{sec: many idempotents}
In this final section, we prove Theorem~\ref{thm: main 1 intro}. 
As noted in the introduction, the implication $\eqref{item: 1} \Rightarrow \eqref{item: 6}$ was established in 
\cite[Theorem~2.5]{GlanserMegrelishvili}; although the argument there concerns the classical Floyd system (or, more precisely, Auslander's variant of it; see Remark~\ref{rem: classical Floyd system}), it carries over verbatim to all Floyd-Auslander systems with finite $\Lambda:=\{n\in\N:|Q(n)|\ge 2\}$.
Indeed, \cite[Theorem~2.5]{GlanserMegrelishvili} shows that finiteness of $Q(n)$ already implies that $(X,T)$ is tame$_1$---that is, not only does $|E(X,T)|\leq 2^{\aleph_0}$ hold, but the space $E(X,T)$ is also first-countable.

Hence, this section is about establishing the implication 
$\eqref{item: 6} \Rightarrow \eqref{item: 4}$, which is the technically most involved part of Theorem~\ref{thm: main 1 intro}.
Using Proposition~\ref{prop: star implies main}, we do this by showing that Floyd-Auslander systems with uncountably many interval fibres satisfy property $(\ast)$.
The proof is divided into the following three cases, which are separately dealt with in Lemmas \ref{thm:case1}, \ref{thm:case2}, and \ref{thm:case3}, respectively.
It is not hard to see but important to note that any Floyd-Auslander system with uncountably many non-invertible fibres (i.e.\ any system where $\Lambda$ is infinite) is covered by (at least) one of the following cases.
\begin{enumerate}[{(A)}]
    \item  There is a strictly increasing sequence $\{n_\ell\}_{\ell\in\N}\subset \L$ such that for each $\ell\in\N$, there exist $a_0(n_\ell)\neq a_1(n_\ell)\in Q(n_\ell)$ and $\Delta(n_\ell)\in [p_{n_\ell}]$ such that $ a_0(n_\ell)+\Delta(n_\ell)\pmod {p_{n_\ell}}\in H_{0}(n_\ell)$ and
$a_1(n_\ell)+\Delta(n_\ell)\pmod {p_{n_\ell}}\in H_{2}(n_\ell)$.\label{item: A}
\medskip
\item There is a strictly increasing sequence $\{n_\ell\}_{\ell\in\N}\subset \L$ such that for $\ell\in\N$ and $a_0(n_\ell)\neq a_1(n_\ell)\in Q(n_\ell)$,
    $a_0(n_\ell)+j\pmod {p_{n_\ell}}\in H_i(X)$ if and only if $a_1(n_\ell)+j\pmod {p_{n_\ell}}\in H_i(X)\text{ for each }j\in [p_{n_\ell}]\text{ and }i\in\{0,2\}.$\label{item: B}
\medskip
    \item There is a strictly increasing sequence $\{n_\ell\}_{\ell\in\N}\subset \L$ such that for each $\ell\in\N$, there exist $a_0(n_\ell)\neq a_1(n_\ell)\in Q(n_\ell)$ and $\Delta(n_\ell)\in [p_{n_\ell}]$ such that $Q(n_\ell)\ni a_0(n_\ell)+\Delta(n_\ell)< p_{n_\ell}$ and
$a_1(n_\ell)+\Delta(n_\ell)\pmod {p_{n_\ell}}\in H_{0}(n_\ell)\cup H_{2}(n_\ell)$.\label{item: C}
\end{enumerate}

In all of the following, $\mathcal T$ denotes an uncountable \choice \ of the binary full shift $(\{0,1\}^\N,\sigma)$, see Lemma~\ref{lem: uncountable choice domains}.
We will be dealing with nets indexed by elements of
$(0,1)\times [\mathcal T]^{<\omega}$, where $[\mathcal T]^{<\omega}$ denotes the finite subsets of $\mathcal T$.
When doing so, we consider $(0,1)\times [\mathcal T]^{<\omega}$ equipped with a direction $\preceq$ such that
$(\eps;x^1,\ldots,x^n)\preceq (\eps';y^1,\ldots,y^{k})$
whenever $\eps\geq \eps'$ and $\{x^1,\ldots,x^n\}\ssq \{y^1,\ldots,y^{k}\}$.

\begin{lem}\label{thm:case1}
If $(X,T)$ is a minimal Floyd-Auslander system satisfying \eqref{item: A}, then $(X,T)$ satisfies $(\ast)$. 
\end{lem}
\begin{proof}
    Consider $\{n_\ell\}$ as in \eqref{item: A}, and let $\{m_\ell\}_{\ell\in\N}\ssq\N$ be a strictly increasing sequence such that
$p_{m_\ell}-1\in H_0(m_\ell)\cup H_2(m_\ell)$ for all $\ell\in\N$ (see Definition~\ref{defn: Floyd-Auslander system}).
Since $(X,T)$ is minimal, there exist strictly increasing sequences $\{k_\ell\}_{\ell\in\N}$ and $\{k'_\ell\}_{\ell\in\N}$ with $H_2(k_\ell)\neq\emptyset$ and $H_0(k'_\ell)\neq\emptyset$ for all $\ell\in\N$.
Our assumptions further guarantee that there is a strictly increasing sequence $\{q_\ell\}_{\ell\in\N}$ with $|Q(q_\ell)|\ge 2$.
In the following, we may assume without loss of generality (possibly after passing to subsequences) that
for each $\ell\in\N$
\[
n_\ell< m_\ell<k_\ell<k'_\ell <q_\ell<n_{\ell+1},
\]
and that for each $\ell$ there is $a(n_\ell+1)\in Q(n_\ell+1)$ with\footnote{
Here, we make use of the standing assumption $|Q(n)|<p_n$ introduced at the end of Section~\ref{sec: Floyd-Auslander systems}.
Of course, it may well be that we can only arrange that
$a(n_\ell+1)+1 \pmod{p_{n_\ell+1}}\in H_2(n_\ell+1)$ holds
for infinitely many $\ell$.
In that case, the present proof remains almost entirely the same, except that we would have to verify (2) of property~$(\ast)$ instead of (1).\label{footnote 1}}
\begin{equation}\label{eq:small assumption}
a(n_\ell+1)+1 \pmod{p_{n_\ell+1}}\in H_0(n_\ell+1).
\end{equation}
Given $\ell\in\N$, fix some some $\xi_{q_\ell}^1,\xi_{q_\ell}^2\in Q(q_\ell)$ with
$0\le \xi_{q_\ell}^1<\xi_{q_\ell}^2\le p_{q_\ell}-1$.
Fix $a(i)\in Q(i)$ for all $i\in \N\setminus\big(\{n_\ell\}_{\ell\in\N}\cup \ \{n_\ell+1\}_{\ell\in\N}\cup \{q_\ell\}_{\ell\in\N}\big)$
and, for each $\ell$, let $a(n_\ell+1)\in Q(n_\ell+1)$ be as in \eqref{eq:small assumption}.
Now, for $x\in\mc T$ (with $\mc T$ as above), let $\alpha^x\in\Sigma_{(p_n)}$ be defined by
$\alpha^x_i=a(i)$ for all $i\in \N\setminus\big(\{n_\ell\}_{\ell\in\N}\cup \{q_\ell\}_{\ell\in\N}\big)$,
$\alpha^x_{q_\ell}=\xi_{q_\ell}^1$, and $\alpha^x_{n_\ell}=a_{x_\ell}(n_\ell)$ for all $\ell\in\N$.

We will show that $(X,T)$ satisfies (1) in property~ $(\ast)$ with $a=1/4$ and with $\mc A$ taken to be the (necessarily uncountable) family $\{\alpha^x:\,x\in\mc T\}$.
Note that by construction, $\mc A$ is \good.

Consider $B\ssq \mathcal A$.
Given $\eps>0$, and pairwise distinct $\alpha^{x^1},\ldots,\alpha^{x^k}\in \mathcal A\setminus B$ and $\alpha^{x^{k+1}},\ldots,\alpha^{x^n}\in B$,
 choose $m=m(\eps)\in \N$ such that $2^{-m}<\eps$; for $j=1,2,\ldots,k$, let $\phi^j=(0,0,\ldots,0)\in\{0,1\}^m$ and for  $j=k+1,k+2,\ldots,n$, let $\phi^j=(1,0,\ldots,0)\in\{0,1\}^m$.
  Let $m<\tau_1<\tau_2<\ldots<\tau_m$ be $m$-realising for $x^1,\ldots, x^n$ and $\phi^1,\ldots,\phi^n$.
With $\Delta(n_\ell)$ ($\ell\in \N$) as in \eqref{item: A}, we define 
 \[     t_{\eps;x^1,\ldots,x^n}=\underbrace{\overbrace{00\ldots00\Delta(n_{\tau_{1}})}^{n_{\tau_{1}}}\overbrace{00\ldots00\Delta(n_{\tau_2})}^{n_{\tau_2}-n_{\tau_{1}}}00\ldots00\Delta(n_{\tau_m})}_{n_{\tau_m}}000\ldots
 \]
By definition of $t_{\varepsilon; x^1, \ldots, x^n}$ and the assumptions on $\{m_\ell\}$, we have, for $j=1,\ldots,n$, that
\begin{equation}\label{eq:20255211345}
\begin{split}
    &(\a^{x^j}+t_{\eps;x^1,\ldots,x^n})_i\in Q(i) \text{ for } i\in\N\setminus\{n_{\tau_s},n_{\tau_s}+1,\ldots,m_{\tau_s}\}_{s=1}^m;\\
    & (\a^{x^j}+t_{\eps;x^1,\ldots,x^n})_i\in H_{2\phi^j_s}\text{ for }i\in \{n_{\tau_s}\} _{s=1}^m;\\
    &(\a^{x^j}+t_{\eps;x^1,\ldots,x^n})_i\in
    \begin{cases}
        Q(i) & \text{if }c_{i-1}=0,\\
        H_{0}(i) & \text{otherwise,}
    \end{cases}\text{ for } i\in\{n_{\tau_s}+1\}_{s=1}^m;\\
    &(\a^{x^j}+t_{\eps;x^1,\ldots,x^n})_i\in
    \begin{cases}
        Q(i) & \text{if }c_{i-1}=0,\\
        Q(i)\cup H_{0}(i)\cup H_2(i) & \text{otherwise,}
    \end{cases}
    \text{ for } i\in\{n_{\tau_s}+2,\ldots,m_{\tau_s}\}_{s=1}^m.
\end{split}
\end{equation}
Here, $c_{i-1}$ denotes the carry from position $i-1$ to position $i$.
Note that the second alternative in the third line follows from \eqref{eq:small assumption}.
Note further that if there is no carry from position
$n_{\tau_s}$, then there is no carry from any of the positions $n_{\tau_s}+1,\ldots,m_{\tau_s}-1$ ($s=1,\ldots,m$).

The above defines a net $\{t_{\eps;x^1,\ldots,x^n}\}$ (indexed through $(\eps,\{x^1,\ldots,x^n\})\in (0,1)\times [\mathcal T]^{<\omega}$).
Observe that $\{t_{\varepsilon; x^1, \ldots, x^n}\}$ converges to $0$ in $\Sigma_{(p_n)}$.
Consequently,  for any accumulation point $f$ of the net $\{T^{t_{\eps;x^1,\ldots,x^{n}}}\}\ssq E(X)$,
we have $f(L_\alpha) \subset L_\alpha$  for every $\alpha \in \Sigma_{(p_n)}$.
Moreover, note that $m \to \infty$ as $\varepsilon \to 0$.
Hence, the second line in \eqref{eq:20255211345} gives for any $(\alpha, y), (\alpha, y') \in X$ ($\alpha \in \mathcal{A}$),
\begin{equation}\label{eq:20255211438}
    d(f(\alpha, y), f(\alpha, y')) = 0.
\end{equation}
Now consider the compact right topological semigroup $\langle f \rangle\ssq E(X, T)$ generated by $f$.
By Lemma~\ref{lem:existence of idempotent}, there exists an idempotent $\tilde{f} \in \langle f \rangle$.
By \eqref{eq:20255211438}, for all $\alpha \in \mathcal{A}$, we have
$\tilde{f}(L_\alpha) = f(L_\alpha)$.
In the following, we may hence assume without loss of generality that $f$ itself is an idempotent.

Now, given such $f$ and $\alpha^x\in \mathcal A$
as well as $\eps>0$,
choose $\tau=t_{\eps';x,x^1,\ldots,x^{n}}$ with $\eps'<\eps$ such that \begin{equation}\label{eq:1455}
    d(f(\alpha^x,1),T^\tau{(\alpha^x,1)})<\eps.
\end{equation}
By  \eqref{eq:20255211345}, we have
\begin{equation*}%\label{eq:case1.1}
    \pi_2\circ T^\tau (\a^x,1)=\l_0\circ\l_0\circ\l_1^x\circ\ldots\circ\l^x_{N_x}(1)\qquad (\a^x\in\mathcal{A}\setminus B),
\end{equation*}
and
\begin{equation*}%\label{eq:case1.2}
     \pi_2\circ T^\tau (\a^x,1)=\l_2\circ\l_0\circ\l_1^x\circ\ldots\circ\l^x_{N_x}(1)\qquad (\a^x\in  B),
\end{equation*}
where $\l^x_j\in \{\l_0,\l_2\}$ and $j=1,2,\ldots,N_x$ for some $N_x\ge m$.
Thus, if $\a^x\in\mathcal{A}\setminus B$, then for any $z\in[0,1]$,
\[|\pi_2\circ f(\alpha^x,z)-[0,1/4]|\leq d(f(\alpha^x,z),f(\a^x,1))+d(f(\alpha^x,1),T^{\tau}(\a^x,1))+|\pi_2\circ T^{\tau}(\a^x,1)-[0,1/4]|\leq \eps,\]
where $|t-[p,q]|=\min\{|t-s|:s\in[p,q]\}$.
As $\eps>0$ was arbitrary,
we conclude $\pi_2\circ f(\alpha^x,z)\in [0,1/4]$.
Similarly, if $\a^x\in  B$, then for any $z\in[0,1]$,
$|\pi_2\circ f(\alpha^x,z)-[1/2,1]|\leq \eps$.
This shows ($\ast$) (with $B_2=\emptyset$ for each $B\ssq \mathcal A$).
\end{proof}

\begin{lem}\label{thm:case2}
If $(X,T)$ is a minimal Floyd-Auslander system satisfying \eqref{item: B}, then $(X,T)$ satisfies $(\ast)$.
\end{lem}
\begin{proof}
Fix an infinite sequence $\{n_\ell\}_{\ell\in\N}$ as in \eqref{item: B}, and let
$\{m_\ell\}_{\ell\in\N}$, $\{k_\ell\}_{\ell\in\N}$, $\{k_\ell'\}_{\ell\in\N}$, and $\{q_\ell\}_{\ell\in\N}$ be strictly increasing sequences, together with elements $\xi_{q_\ell}^1,\xi_{q_\ell}^2\in Q(q_\ell)$ satisfying $0\le \xi_{q_\ell}^1<\xi_{q_\ell}^2<p_{q_\ell}$, defined similarly as in the proof of Lemma~\ref{thm:case1}.
In particular,
\[
n_\ell< m_\ell<k_\ell<k_\ell'<q_\ell<n_{\ell+1}.
\]
Recall our standing assumption
\begin{equation}\label{eq:20255251429}
H_0(n)\cup H_2(n)\neq\emptyset \qquad \text{for all } n\in\N .
\end{equation}

For each $i\in \N\setminus\big(\{n_\ell\}_{\ell\in\N}\cup \{q_\ell\}_{\ell\in\N}\big)$, fix some $a(i)\in Q(i)$.
In view of \eqref{eq:20255251429}, we may assume without loss of generality that
$a(n_\ell+1)+1\pmod{p_{n_\ell+1}}\in H_0(n_\ell+1)\cup H_2(n_\ell+1)$ for all $\ell\in\N$.
For $x\in\mc T$ (with $\mc T$ as above), define $\alpha^x\in\Sigma_{(p_n)}$ by
\[
\alpha^x_i=a(i)\quad\text{for all }i\in \N\setminus\big(\{n_\ell\}_{\ell\in\N}\cup \{q_\ell\}_{\ell\in\N}\big),\qquad
\alpha^x_{q_\ell}=\xi_{q_\ell}^1,\qquad
\alpha^x_{n_\ell}=a_{x_\ell}(n_\ell)\quad(\ell\in\N).
\]
Let $\mathcal A:=\{\alpha^x:\,x\in\mc T\}$ denote the uncountable family of all such sequences.
By construction, $\mathcal A$ is \good.
Moreover, by an argument similar to that in the proof of Lemma~\ref{lem:existence of minimal}, for each $\ell\in \N,$ we may choose $u^\ell_0\in [p_{k_\ell}],\ldots,u^\ell_{q_\ell-k_\ell}\in [p_{q_\ell}]$ such that \begin{equation}\label{eq:2025991926}
    \gamma_\ell=\overbrace{00\ldots00}^{k_\ell-1}\overbrace{u^\ell_0u^\ell_1\ldots u^\ell_{q_\ell-k_\ell}}^{q_\ell-k_\ell+1}00\ldots,
\end{equation}
satisfies,  for each $\a\in\mathcal{A}$,
\begin{align}\label{eq: 47}(\a+ \gamma_\ell)_n\in\begin{cases}
H_2(n) & \text{ for } n=k_\ell,\\
Q(n) & \text{ otherwise;}
\end{cases}
 \end{align} and
\begin{align}\label{eq: almost all coordinates unchanged}
    \a_n=(\a+\gamma_\ell)_n \text{ for all }n\in \N\setminus \{k_\ell,k_{\ell+1},\ldots,q_\ell\}.
\end{align}

To prove that $(X,T)$ satisfies property~$(\ast)$, we consider the following two cases, which---after going over to subsequences if needed---cover all possibilities.

\begin{itemize}
    \item[\textbf{Case 1.}] There exists $i\in\{0,2\}$ such that
\[a(n_\ell+1)+1\bmod p_{n_\ell+1}\in H_i(n_\ell+1)\text{ and }H_{2-i}(n_\ell)\neq\emptyset\text{ for all }\ell\in\N.\]
    \item[\textbf{Case 2.}] There exists $i\in\{0,2\}$ such that
\[a(n_\ell+1)+1\bmod p_{n_\ell+1}\in H_i(n_\ell+1)\text{ and }H_{2-i}(n_\ell)=\emptyset\text{ for all }\ell\in\N.\]
\end{itemize}
We now treat each case separately.

\medskip
\noindent\textbf{Case 1.} Similar as in the proof of Lemma~\ref{thm:case1} (see also footnote~\ref{footnote 1}), we may assume without loss of generality that
$i=0$, that is,
\begin{equation*}
     a(n_\ell+1)+1\bmod p_{n_\ell+1}\in H_0(n_\ell+1)\text{ and }H_{2}(n_\ell)\neq\emptyset \text{ for all } \ell\in\N.
\end{equation*}
Observe that due to assumption \eqref{item: B}, for each $\ell\in\N$, we may (and do) choose $a_0(n_\ell)\neq a_1(n_\ell)\in Q(n_\ell)$ and $\D(n_\ell)$ such that $H_2(n_\ell)\ni a_0(n_\ell)+\D(n_\ell)\ge p_{n_\ell}$ and $H_2(n_\ell)\ni a_1(n_\ell)+\D(n_\ell)< p_{n_\ell}$.

Pick some $B\subset \mathcal{A}$.
 Given $\eps>0$, and distinct $\alpha^{x^1},\ldots,\alpha^{x^k}\in \mathcal A\setminus B$ and $\alpha^{x^{k+1}},\ldots,\alpha^{x^n}\in B$,
 choose $m\in \N$ such that $2^{-m}<\eps$;
 pick $\phi^j=(0,1,\ldots,1)\in\{0,1\}^m$ for $j=1,2,\ldots,k$ and $\phi^j=(1,1,\ldots,1)\in\{0,1\}^m$ for $j=k+1,k+2,\ldots,n$.
Let $m<\tau_1<\tau_2<\ldots<\tau_m$ be $m$-realising for $x^1,\ldots, x^n$ and $\phi^1,\ldots,\phi^n$ and
 set
 \[     t_{\eps;x^1,\ldots,x^n}=\underbrace{\overbrace{0000\Delta(n_{\tau_{1}})}^{n_{\tau_{1}}}\overbrace{00\ldots00\Delta(n_{\tau_2})}^{n_{\tau_2}-n_{\tau_{1}}}000\ldots00\Delta(n_{\tau_m})}_{n_{\tau_m}}.
 \]
Note that for $j=1,2\ldots,k$
\begin{equation}\label{eq:20255211345new}
\begin{split}
    &(\a^{x^j}+t_{\eps;x^1,\ldots,x^n})_i\in Q(i), \text{ for } i\in\N\setminus(\{n_{\tau_s}\}_{s=1}^m\cup\{n_{\tau_1}+1,\ldots,m_{\tau_1}\}),\\
    & (\a^{x^j}+t_{\eps;x^1,\ldots,x^n})_i\in H_{2}(i),\text{ for }i\in\{n_{\tau_s}\}_{s=1}^m\\
    &(\a^{x^j}+t_{\eps;x^1,\ldots,x^n})_i\in H_{0}(i),\text{ for $i=n_{\tau_1}+1$},\\
    &(\a^{x^j}+t_{\eps;x^1,\ldots,x^n})_i\in H_{0}(i)\cup H_2(i)\cup Q(i),\text{ for $i\in\{n_{\tau_1}+2,\ldots,m_{\tau_1}\}$},
\end{split}
\end{equation}
while for $j=k+1,k+2,\ldots,n$,
\begin{equation}\label{eq:20255211345new2}
\begin{split}
    &(\a^{x^j}+t_{\eps;x^1,\ldots,x^n})_i\in Q(i), \text{ for } i\in\N\setminus\{n_{\tau_s}\}_{s=1}^m,\\
    & (\a^{x^j}+t_{\eps;x^1,\ldots,x^n})_i\in H_{2}(i),\text{ for }i\in\{n_{\tau_s}\}_{s=1}^m.
\end{split}
\end{equation}

As in the proof of Lemma~\ref{thm:case1}, we let \( f \) be an idempotent in the semigroup generated by an accumulation point of the net
$
\{T^{t_{\varepsilon; x^1,\ldots,x^n}}\}_{(\varepsilon; x^1,\ldots,x^n) \in (0,1) \times [\mathcal{T}]^{<\omega}}
\ssq  E(X)$.
Using \eqref{eq:20255211345new} and \eqref{eq:20255211345new2}, we can---following a similar reasoning as in Lemma~\ref{thm:case1}---show that for each \( z \in [0,1] \)
\begin{align*}
\pi_2 \circ f(\alpha^x, z) \in [0, 3/4] \ \ (\alpha^x\in \mc A\setminus B) \qquad \text{and} \qquad
\pi_2 \circ f(\alpha^x, z) = 1 \ \ ( \alpha^x \in B).
\end{align*}
Hence, with \( a = 3/4 \) and \( B_2 = \emptyset \) for each $B\ssq \mc A$, the system \( (X,T) \) satisfies (1) in property \((*)\).

\medskip

\noindent\textbf{Case 2.}
Similar as in the previous case, we may assume without loss of generality that
$i=0$, that is,
\begin{equation*}
     a(n_\ell+1)+1\bmod p_{n_\ell+1}\in H_0(n_\ell+1)\text{ and }H_{2}(n_\ell)=\emptyset \text{ for all } \ell\in\N.
\end{equation*}
By \eqref{eq:20255251429}, this implies
\begin{equation}\label{eq:20255251514}
    H_{0}(n_\ell)\neq\emptyset \text{ for all } \ell\in\N.
\end{equation}
Due to assumption \eqref{item: B}, for each $\ell \in \N$, we can choose $a_0(n_\ell)\neq a_1(n_\ell)\in Q(n_\ell)$ and $\D(n_\ell)$ such that $H_0(n_\ell)\ni a_0(n_\ell)+\D(n_\ell)\ge p_{n_\ell}$ and $H_0(n_\ell)\ni a_1(n_\ell)+\D(n_\ell)< p_{n_\ell}$.

Pick some $B\subset \mathcal{A}$.
 Given $\eps>0$, and distinct $\alpha^{x^1},\ldots,\alpha^{x^k}\in \mathcal A\setminus B$ and $\alpha^{x^{k+1}},\ldots,\alpha^{x^n}\in B$,
 choose $m\in \N$ such that $2^{-m}<\eps$;
 pick $\phi^j=(0,1,\ldots,1)\in\{0,1\}^m$ for $j=1,2,\ldots,k$ and $\phi^j=(1,1,\ldots,1)\in\{0,1\}^m$ for $j=k+1,k+2,\ldots,n$.
Let  $m<\tau_1<\tau_2<\ldots<\tau_m$ be $m$-realising for $x^1,\ldots, x^n$ and $\phi^1,\ldots,\phi^n$ and
 set
 \[     t_{\eps;x^1,\ldots,x^n}=\underbrace{\overbrace{0000\Delta(n_{\tau_{1}})}^{n_{\tau_{1}}}\overbrace{00\ldots00\Delta(n_{\tau_2})}^{n_{\tau_2}-n_{\tau_{1}}}000\ldots00\ldots00\Delta(n_{\tau_m})}_{n_{\tau_m}}+\sum_{i=1}^m\gamma_{\tau_i},
 \]
with $\gamma_i$ as in \eqref{eq:2025991926}.
Using \eqref{eq: 47} and \eqref{eq: almost all coordinates unchanged}, we have that for $j=1,2,\ldots,k$,
\begin{equation}\label{eq:20255211345neww}
\begin{split}
    &(\a^{x^j}+t_{\eps;x^1,\ldots,x^n})_i\in Q(i), \text{ for } i\in\N\setminus(\{n_{\tau_s}\}_{s=1}^m\cup\{n_{\tau_1}+1,\ldots,m_{\tau_1}\}\cup\{k_{\tau_s}\}_{s=1}^m),\\
    & (\a^{x^j}+t_{\eps;x^1,\ldots,x^n})_i\in H_{0}(i),\text{ for }i\in\{n_{\tau_s}\}_{s=1}^m\cup\{n_{\tau_1}+1\},\\
    &(\a^{x^j}+t_{\eps;x^1,\ldots,x^n})_i\in H_{2}(i),\text{ for $i\in\{k_{\tau_s}\}_{s=1}^m$},\\
    &(\a^{x^j}+t_{\eps;x^1,\ldots,x^n})_i\in H_0(i)\cup H_{2}(i)\cup Q(i),\text{ for $i\in\{n_{\tau_1}+2,\ldots,m_{\tau_1}\}$},
\end{split}
\end{equation}
and for $j=k+1,k+2,\ldots,n$,
\begin{equation}\label{eq:20255211345neww2}
\begin{split}
    &(\a^{x^j}+t_{\eps;x^1,\ldots,x^n})_i\in Q(i), \text{ for } i\in\N\setminus(\{n_{\tau_s}\}_{s=1}^m\cup\{k_{\tau_s}\}_{s=1}^m),\\
     & (\a^{x^j}+t_{\eps;x^1,\ldots,x^n})_i\in H_{0}(i),\text{ for }i\in\{n_{\tau_s}\}_{s=1}^m\\
    & (\a^{x^j}+t_{\eps;x^1,\ldots,x^n})_i\in H_{2}(i),\text{ for }i\in\{k_{\tau_s}\}_{s=1}^m.
\end{split}
\end{equation}
Using \eqref{eq:20255211345neww} and \eqref{eq:20255211345neww2}, we can---arguing as before---show that there is an idempotent $f$ such that for every $z\in [0,1]$,
\begin{align*}
\pi_2 \circ f(\alpha^x, z) \in [0, 1/4] \ \  (\alpha^x \in \mathcal{A} \setminus B) \qquad \text{and} \qquad
\pi_2 \circ f(\alpha^x, z) = 1/3 \ \ (\alpha^x \in B).
\end{align*}

This shows that \( (X,T) \) satisfies (1) in property \((*)\)
with \( a = 1/4 \) and \( B_2 = \emptyset \) for each $B\ssq \mc A$ and thus,
finishes the proof.
\end{proof}

\begin{lem}\label{thm:case3}
 If $(X,T)$ is a minimal Floyd-Auslander system satisfying \eqref{item: C}, then $(X,T)$ satisfies $(\ast)$.
\end{lem}
\begin{proof}
Fix an infinite sequence \( \{n_\ell\}_{\ell \in \mathbb{N}} \) as in \eqref{item: C}, and let
$\{m_\ell\}_{\ell\in\N}$, $\{k_\ell\}_{\ell\in\N}$, $\{k_\ell'\}_{\ell\in\N}$, and $\{q_\ell\}_{\ell\in\N}$ be strictly increasing sequences together with elements $\xi_{q_\ell}^1,\xi_{q_\ell}^2\in Q(q_\ell)$ satisfying $0\le \xi_{q_\ell}^1<\xi_{q_\ell}^2<p_{q_\ell}$,  defined similarly as in the proof of Lemma~\ref{thm:case1}.
In particular,
\[
n_\ell< m_\ell<k_\ell<k'_\ell<q_\ell<n_{\ell+1}.
\]
Passing to a subsequence, we may assume without loss of generality that\footnote{The case where
$a_1(n_\ell) + \Delta(n_\ell) \pmod {p_{n_\ell}} \in H_2(n_\ell)$ for all $\ell \in \mathbb{N}$ is similar.}
\[
a_1(n_\ell) + \Delta(n_\ell) \pmod {p_{n_\ell}} \in H_0(n_\ell) \quad \text{for all } \ell \in \mathbb{N}.
\]
% Furthermore, we assume that
% $
% p_{n_\ell} \le a_1(n_\ell) + \Delta(n_\ell)
% $---the case where
% $
% p_{n_\ell} > a_1(n_\ell) + \Delta(n_\ell)
% $
% can be dealt with through a similar (slightly simpler) argument.

For each $i\in \N\setminus\big(\{n_\ell\}_{\ell\in\N}\cup \{q_\ell\}_{\ell\in\N}\big)$, we now fix some $a(i)\in Q(i)$.
For $x\in\mc T$ (with $\mc T$ as above), define $\alpha^x\in\Sigma_{(p_n)}$ by
\[
\alpha^x_i=a(i)\quad\text{for all }i\in \N\setminus\big(\{n_\ell\}_{\ell\in\N}\cup \{q_\ell\}_{\ell\in\N}\big),\qquad
\alpha^x_{q_\ell}=\xi_{q_\ell}^1,\qquad
\alpha^x_{n_\ell}=a_{x_\ell}(n_\ell)\quad(\ell\in\N).
\]
Let $\mathcal A:=\{\alpha^x:\,x\in\mc T\}$ denote the family of all such sequences.
By construction, $\mathcal A$ is \good.

Consider $B\subset\mathcal{A}$.
Given $\eps>0$, and distinct $\alpha^{x^1},\ldots,\alpha^{x^k}\in \mathcal A\setminus B$ and $\alpha^{x^{k+1}},\ldots,\alpha^{x^n}\in B$,
 choose $m\in \N$ such that $2^{-m}<\eps$; for $j=1,2,\ldots,k$,
 pick $\phi^j=(1,1,\ldots,1)\in\{0,1\}^m$ and for $j=k+1,k+2,\ldots,n$, set $\phi^j=(0,0,\ldots,0)\in\{0,1\}^m$.

 Let $m<\tau_1<\tau_2<\ldots<\tau_m$ be $m$-realising for $x^1,\ldots, x^n$ and $\phi^1,\ldots,\phi^n$ and
 set
 \[     t_{\eps;x^1,\ldots,x^n}=\underbrace{\overbrace{0000\Delta(n_{\tau_{1}})}^{n_{\tau_{1}}}\overbrace{00\ldots00\Delta(n_{\tau_2})}^{n_{\tau_2}-n_{\tau_{1}}}000\ldots00\Delta(n_{\tau_m})}_{n_{\tau_m}}00\ldots
 \]

% Note that as $|Q(n)|<p_n$ for each $n$, we may assume without loss of generality---going over to subsequences if needed---to be in one of the following two cases.
% \begin{enumerate}
%     \item[\textbf{Case 1.}] $a(n_\ell+1)+1 \pmod{p_{n_\ell+1}}\in H_0(n_\ell+1)$ for all $\ell \in \mathbb{N}$.
%     \item[\textbf{Case 2.}] $a(n_\ell+1)+1 \pmod{p_{n_\ell+1}}\in H_2(n_\ell+1)$ for all $\ell \in \mathbb{N}$.
% \end{enumerate}

% We now consider each case separately.

% \noindent\textbf{Case 1.}
% In this case, 
Note that for $j=1,2,\ldots,k$
\begin{equation}\label{eq:20255211345new4}
\begin{split}
    &(\a^{x^j}+t_{\eps;x^1,\ldots,x^n})_i\in Q(i), \text{ for } i\in\N\setminus\left(\{n_{\tau_s},n_{\tau_s}+1,\ldots,m_{\tau_s}\}_{s=1}^m\right),\\
    & (\a^{x^j}+t_{\eps;x^1,\ldots,x^n})_i\in H_{0}(i),\text{ for }i\in\{n_{\tau_s}\}_{s=1}^m,\\
    & (\a^{x^j}+t_{\eps;x^1,\ldots,x^n})_i\in H_{0}(i)\cup H_2(i)\cup Q(i),\text{ for }i\in\{n_{\tau_s}+1,n_{\tau_s}+2,\ldots,m_{\tau_s}\}_{s=1}^m,
\end{split}
\end{equation}
while for $j=k+1,k+2,\ldots,n$,
\begin{equation}\label{eq:20255211345neww4}
(\a^{x^j}+t_{\eps;x^1,\ldots,x^n})_i\in Q(i), \text{ for } i\in\N.
\end{equation}

Using \eqref{eq:20255211345new4} and \eqref{eq:20255211345neww4}, we can---similarly as before---show that there is an idempotent $f$ such that for each $z\in [0,1]$
\begin{align*}
\pi_2 \circ f(\alpha^x, z) \in [0,1/2] \quad (\alpha^x \in \mathcal{A} \setminus B )
\qquad \text{and} \qquad \pi_2 \circ f(\alpha^x, z) = z \quad (\alpha^x \in B).
\end{align*}
That is, \( (X,T) \) satisfies (1) of property \((*)\) with
\( a = 1/2 \) and \( B_1 = \emptyset \) for each $B\ssq \mc A$.
%
% \noindent\textbf{Case 2.}
% Now, observe that for $j=1,2\ldots,k$
% \begin{equation}\label{eq:20255211345new5}
% \begin{split}
%     &(\a^{x_j}+t_{\eps;x^1,\ldots,x^n})_i\in Q(i), \text{ for } i\in\N\setminus\left(\{n_{\tau_s},n_{\tau_s}+1,\ldots,m_{\tau_s}\}_{s=1}^m\right),\\
%     & (\a^{x_j}+t_{\eps;x^1,\ldots,x^n})_i\in H_{0}(i),\text{ for }i\in\{n_{\tau_s}\}_{s=1}^m,\\
%      & (\a^{x_j}+t_{\eps;x^1,\ldots,x^n})_i\in H_{2}(i),\text{ for }i\in \{n_{\tau_s}+1\}_{s=1}^m,\\
%      & (\a^{x_j}+t_{\eps;x^1,\ldots,x^n})_i\in H_0(i)\cup H_{2}(i)\cup Q(i),\text{ for }i\in \{n_{\tau_s}+2,\ldots,m_{\tau_s}\}_{s=1}^m
% \end{split}
% \end{equation}
% while for $j=k+1,k+2\ldots,n$, we again have
% \begin{equation}\label{eq:20255211345neww5}
% (\a^{x_j}+t_{\eps;x^1,\ldots,x^n})_i\in Q(i), \text{ for } i\in\N.
% \end{equation}
%
% Using \eqref{eq:20255211345new5} and \eqref{eq:20255211345neww5}, we can show---just like before---that for each \( z \in [0,1] \)
% \begin{align*}
% \pi_2 \circ f(\alpha^x, z) \in [0,1/2] \quad ( \alpha^x \in \mathcal{A} \setminus B ) \qquad \text{and} \qquad \pi_2 \circ f(\alpha^x, z) = z \quad ( \alpha^x \in B )
% \end{align*}
% That is, \( (X,T) \) satisfies (1) of property \((*)\) with
% \( a = 1/2 \) and \( B_1 = \emptyset \) for each $B\ssq \mc A$.
% This finishes the proof.
\end{proof}

\bibliographystyle{acm}
\bibliography{ref}

@article {HLSY2021,
    AUTHOR = {Huang, Wen and Lian, Zhengxing and Shao, Song and Ye,
              Xiangdong},
     TITLE = {Minimal systems with finitely many ergodic measures},
   JOURNAL = {J. Funct. Anal.},
  FJOURNAL = {Journal of Functional Analysis},
    VOLUME = {280},
      YEAR = {2021},
    NUMBER = {12},
     PAGES = {Paper No. 109000, 42},
      ISSN = {0022-1236},
   MRCLASS = {37A25 (37A15 37B05 37B40)},
  MRNUMBER = {4234226},
MRREVIEWER = {Yunping Wang},
       DOI = {10.1016/j.jfa.2021.109000},
       URL = {https://doi.org/10.1016/j.jfa.2021.109000},
}

@article{Donoso2014,
 author = {Donoso, Sebasti{\'a}n},
 title = {Enveloping semigroups of systems of order {{\(d\)}}},
 fjournal = {Discrete and Continuous Dynamical Systems},
 journal = {Discrete Contin. Dyn. Syst.},
 issn = {1078-0947},
 volume = {34},
 number = {7},
 pages = {2729--2740},
 year = {2014},
 language = {English},
 doi = {10.3934/dcds.2014.34.2729},
 zbMATH = {6265796},
 Zbl = {1330.37006}
}

@article{Glasner1993,
 author = {Glasner, Eli},
 title = {Minimal nil-transformations of class two},
 fjournal = {Israel Journal of Mathematics},
 journal = {Isr. J. Math.},
 issn = {0021-2172},
 volume = {81},
 number = {1-2},
 pages = {31--51},
 year = {1993},
 language = {English},
 doi = {10.1007/BF02761296},
 zbMATH = {410026},
 Zbl = {0780.28009}
}

@article{GlasnerMegrelishvili2006,
 author = {Glasner, E. and Megrelishvili, M.},
 title = {Hereditarily non-sensitive dynamical systems and linear representations},
 fjournal = {Colloquium Mathematicum},
 journal = {Colloq. Math.},
 issn = {0010-1354},
 volume = {104},
 number = {2},
 pages = {223--283},
 year = {2006},
 language = {English},
 doi = {10.4064/cm104-2-5},
 zbMATH = {5012539},
 Zbl = {1094.54020}
}

@article{KerrLi2007,
 author = {Kerr, David and Li, Hanfeng},
 title = {Independence in topological and {{\(C^*\)}}-dynamics},
 fjournal = {Mathematische Annalen},
 journal = {Math. Ann.},
 issn = {0025-5831},
 volume = {338},
 number = {4},
 pages = {869--926},
 year = {2007},
 language = {English},
 doi = {10.1007/s00208-007-0097-z},
 zbMATH = {5199454},
 Zbl = {1131.46046}
}

@article{Glasner2018,
 author = {Glasner, Eli},
 title = {The structure of tame minimal dynamical systems for general groups},
 fjournal = {Inventiones Mathematicae},
 journal = {Invent. Math.},
 issn = {0020-9910},
 volume = {211},
 number = {1},
 pages = {213--244},
 year = {2018},
 language = {English},
 doi = {10.1007/s00222-017-0747-z},
 zbMATH = {6830047},
 Zbl = {1384.54021}
}

@article{KellendonkReem2022,
 author = {Kellendonk, Johannes and Yassawi, Reem},
 title = {The {Ellis} semigroup of bijective substitutions},
 fjournal = {Groups, Geometry, and Dynamics},
 journal = {Groups Geom. Dyn.},
 issn = {1661-7207},
 volume = {16},
 number = {1},
 pages = {29--73},
 year = {2022},
 language = {English},
 doi = {10.4171/GGD/640},
 zbMATH = {7531888},
 Zbl = {1498.37017}
}

@article{Staynova2021,
 author = {Staynova, Petra},
 title = {The {Ellis} semigroup of certain constant-length substitutions},
 fjournal = {Ergodic Theory and Dynamical Systems},
 journal = {Ergodic Theory Dyn. Syst.},
 issn = {0143-3857},
 volume = {41},
 number = {3},
 pages = {935--960},
 year = {2021},
 language = {English},
 doi = {10.1017/etds.2019.75},
 zbMATH = {7454648},
 Zbl = {1487.37015}
}

@misc{LXZ2025,
      title={Independence and mean sensitivity in minimal systems under group actions}, 
      author={Chunlin Liu and Leiye Xu and Shuhao Zhang},
      year={2025},
      eprint={2501.15622},
      archivePrefix={arXiv},
      primaryClass={math.DS},
      url={https://arxiv.org/abs/2501.15622}, 
}

@article {LWX2025,
    AUTHOR = {Liu, Chunlin and Wang, Xiangtong and Xu, Leiye},
     TITLE = {Sequence entropy and {IT}-tuples for minimal group actions},
   JOURNAL = {Adv. Math.},
  FJOURNAL = {Advances in Mathematics},
    VOLUME = {467},
      YEAR = {2025},
     PAGES = {Paper No. 110183, 34},
      ISSN = {0001-8708},
   MRCLASS = {37A25 (37A15 37A35 37B05)},
  MRNUMBER = {4871426},
       DOI = {10.1016/j.aim.2025.110183},
       URL = {https://doi.org/10.1016/j.aim.2025.110183},
}

@article{FuhrmannKellendonkYassawi2024,
 author = {Fuhrmann, Gabriel and Kellendonk, Johannes and Yassawi, Reem},
 title = {Tame or wild {Toeplitz} shifts},
 fjournal = {Ergodic Theory and Dynamical Systems},
 journal = {Ergodic Theory Dyn. Syst.},
 issn = {0143-3857},
 volume = {44},
 number = {5},
 pages = {1379--1417},
 year = {2024},
 language = {English},
 doi = {10.1017/etds.2023.58},
 zbMATH = {7861154},
 Zbl = {1551.37018}
}

@article{Huang2006,
 author = {Huang, Wen},
 title = {Tame systems and scrambled pairs under an {Abelian} group action},
 fjournal = {Ergodic Theory and Dynamical Systems},
 journal = {Ergodic Theory Dyn. Syst.},
 issn = {0143-3857},
 volume = {26},
 number = {5},
 pages = {1549--1567},
 year = {2006},
 language = {English},
 doi = {10.1017/S0143385706000198},
 zbMATH = {5122481},
 Zbl = {1122.37009}
}

@article {Auslander,
    AUTHOR = {Auslander, Joseph},
     TITLE = {Mean-{$L$}-stable systems},
   JOURNAL = {Illinois J. Math.},
  FJOURNAL = {Illinois Journal of Mathematics},
    VOLUME = {3},
      YEAR = {1959},
     PAGES = {566--579},
      ISSN = {0019-2082},
   MRCLASS = {54.82},
  MRNUMBER = {149462},
MRREVIEWER = {J. C. Oxtoby},
       URL = {http://projecteuclid.org/euclid.ijm/1255455462},
}

@book {Auslanderbook,
    AUTHOR = {Auslander, Joseph},
     TITLE = {Minimal flows and their extensions},
    SERIES = {North-Holland Mathematics Studies},
    VOLUME = {153},
      NOTE = {Notas de Matem\'{a}tica [Mathematical Notes], 122},
 PUBLISHER = {North-Holland Publishing Co., Amsterdam},
      YEAR = {1988},
     PAGES = {xii+265},
      ISBN = {0-444-70453-1},
   MRCLASS = {54H20},
  MRNUMBER = {956049},
MRREVIEWER = {M. Rees},
}

@article {FuhrmannDominik,
    AUTHOR = {Fuhrmann, Gabriel and Kwietniak, Dominik},
     TITLE = {On tameness of almost automorphic dynamical systems for
              general groups},
   JOURNAL = {Bull. Lond. Math. Soc.},
  FJOURNAL = {Bulletin of the London Mathematical Society},
    VOLUME = {52},
      YEAR = {2020},
    NUMBER = {1},
     PAGES = {24--42},
      ISSN = {0024-6093},
   MRCLASS = {37B05 (22F05 37A15 37B10)},
  MRNUMBER = {4072030},
MRREVIEWER = {Eusebio Gardella},
       DOI = {10.1112/blms.12304},
       URL = {https://doi.org/10.1112/blms.12304},
}

@article {HaddadJohnson1997,
    AUTHOR = {Haddad, Kamel N. and Johnson, Aimee S. A.},
     TITLE = {Auslander systems},
   JOURNAL = {Proc. Amer. Math. Soc.},
  FJOURNAL = {Proceedings of the American Mathematical Society},
    VOLUME = {125},
      YEAR = {1997},
    NUMBER = {7},
     PAGES = {2161--2170},
      ISSN = {0002-9939},
   MRCLASS = {54H20 (54H15)},
  MRNUMBER = {1372033},
       DOI = {10.1090/S0002-9939-97-03768-4},
       URL = {https://doi.org/10.1090/S0002-9939-97-03768-4},
}

@article {FuhrmannGlasnerJagerOertel2021,
    AUTHOR = {Fuhrmann, G. and Glasner, E. and J\"{a}ger, T. and Oertel, C.},
     TITLE = {Irregular model sets and tame dynamics},
   JOURNAL = {Trans. Amer. Math. Soc.},
  FJOURNAL = {Transactions of the American Mathematical Society},
    VOLUME = {374},
      YEAR = {2021},
    NUMBER = {5},
     PAGES = {3703--3734},
      ISSN = {0002-9947},
   MRCLASS = {37B52 (37B05 37B40 52C23)},
  MRNUMBER = {4237960},
MRREVIEWER = {Samuel Petite},
       DOI = {10.1090/tran/8349},
       URL = {https://doi.org/10.1090/tran/8349},
}

@article{Kellendonk2024,
 author = {Kellendonk, Johannes},
 title = {On non-tameness of the {Ellis} semigroup},
 fjournal = {Ergodic Theory and Dynamical Systems},
 journal = {Ergodic Theory Dyn. Syst.},
 issn = {0143-3857},
 volume = {45},
 number = {11},
 pages = {3419--3429},
 year = {2025},
 language = {English},
 doi = {10.1017/etds.2025.10190},
 zbMATH = {8105087}
}

@book {HindmanStrauss2012,
    AUTHOR = {Hindman, Neil and Strauss, Dona},
     TITLE = {Algebra in the {S}tone-\v{C}ech compactification},
    SERIES = {De Gruyter Textbook},
      NOTE = {Theory and applications,
              Second revised and extended edition [of MR1642231]},
 PUBLISHER = {Walter de Gruyter \& Co., Berlin},
      YEAR = {2012},
     PAGES = {xviii+591},
      ISBN = {978-3-11-025623-9},
   MRCLASS = {54-02 (03E05 22A15 54D35 54H99)},
  MRNUMBER = {2893605},
}

@article{GlanserMegrelishvili,
  title={ Todor\u{c}evi\'{c}' trichotomy and a hierarchy in the class of tame dynamical systems},
  author={Glasner, Eli and Megrelishvili, Michael},
  journal={Transactions of the American Mathematical Society},
  volume={375},
  number={07},
  pages={4513--4548},
  year={2022}
}

@book {Ellis1969,
    AUTHOR = {Ellis, Robert},
     TITLE = {Lectures on topological dynamics},
 PUBLISHER = {W. A. Benjamin, Inc., New York},
      YEAR = {1969},
     PAGES = {xv+211},
   MRCLASS = {54.82},
  MRNUMBER = {267561},
MRREVIEWER = {J.\ Auslander},
}

@article{Floyd1949,
 author = {Floyd, E. E.},
 title = {A nonhomogeneous minimal set},
 fjournal = {Bulletin of the American Mathematical Society},
 journal = {Bull. Am. Math. Soc.},
 issn = {0002-9904},
 volume = {55},
 pages = {957--960},
 year = {1949},
 language = {English},
 doi = {10.1090/S0002-9904-1949-09318-7},
 zbMATH = {3054338},
 Zbl = {0035.36002}
}

@book {MR603625,
    AUTHOR = {Furstenberg, H.},
     TITLE = {Recurrence in ergodic theory and combinatorial number theory},
      NOTE = {M. B. Porter Lectures},
 PUBLISHER = {Princeton University Press, Princeton, NJ},
      YEAR = {1981},
     PAGES = {xi+203},
      ISBN = {0-691-08269-3},
   MRCLASS = {28D05 (10K10 10L10 54H20)},
  MRNUMBER = {603625},
MRREVIEWER = {Michael Keane},
}

@article {MR48467,
    AUTHOR = {Numakura, Katsumi},
     TITLE = {On bicompact semigroups},
   JOURNAL = {Math. J. Okayama Univ.},
  FJOURNAL = {Mathematical Journal of Okayama University},
    VOLUME = {1},
      YEAR = {1952},
     PAGES = {99--108},
      ISSN = {0030-1566},
   MRCLASS = {20.0X},
  MRNUMBER = {48467},
MRREVIEWER = {Melvin Henriksen},
}
\end{document}